\documentclass[a4paper,12pt]{article}

\usepackage[margin=1.2in]{geometry}

\usepackage{amsmath}
\usepackage{algorithm}
\usepackage{algpseudocode}
\usepackage{amsthm}
\usepackage{amsfonts}
\usepackage{comment}
\usepackage{graphicx}
\usepackage[hidelinks]{hyperref}
\usepackage{color}
\usepackage{subcaption}
\usepackage{comment}
\usepackage{makecell}
\usepackage{array}
\usepackage{placeins}

\newtheorem{theorem}{Theorem}
\newtheorem{lemma}[theorem]{Lemma}
\newtheorem{corollary}[theorem]{Corollary}
\newtheorem{remark}[theorem]{Remark}
\newtheorem{assumption}[theorem]{Assumption}

\def\x{{\mathbf{x}}}
\def\e{{\mathbf{e}}}
\def\a{{\mathbf{a}}}
\def\u{{\mathbf{u}}}
\def\v{{\mathbf{v}}}
\def\z{{\mathbf{z}}}
\def\w{{\mathbf{w}}}
\def\y{{\mathbf{y}}}
\def\p{{\mathbf{p}}}
\def\b{{\mathbf{b}}}
\def\X{{\mathbf{X}}}

\def\A{{\mathbf{A}}}

\def\I{{\mathbf{I}}}

\def\V{{\mathbf{V}}}

\def\U{{\mathbf{U}}}
\def\Q{{\mathbf{Q}}}
\def\P{{\mathbf{P}}}

\newcommand{\mX}{\mathcal{X}}

\newcommand{\R}{\mathcal{R}}

\newcommand{\mE}{\mathcal{E}}

\newcommand{\mS}{\mathcal{S}}
\newcommand{\ball}{\mathcal{B}}
\newcommand{\mbS}{\mathbb{S}}

\newcommand{\E}{\mathbb{E}}

\newcommand{\nnz}{\mathrm{nnz}}
\newcommand{\trace}{\mathrm{Tr}}
\newcommand{\rank}{\mathrm{rank}}
\newcommand{\reals}{\mathbb{R}}

\DeclareMathOperator*{\argmin}{arg\,min}

\newcommand{\supp}{\mathop{\mbox{\rm supp}}}

\title{First-Order Sparse Convex Optimization: \\ Better Rates with Sparse Updates}
\author{Dan Garber \\ \small{dangar@technion.ac.il} \\ \small{Faculty of Data and Decision Sciences} \\ \small{Technion - Israel Institute of Technology}}
\date{}

\begin{document}

\maketitle

\begin{abstract}

It was recently established that for convex optimization problems with sparse optimal solutions (be it entry-wise sparsity or matrix rank-wise sparsity), it is possible to design first-order methods with linear convergence rates that depend on an improved mixed-norm condition number of the form $\beta_1 s/\alpha_2$, where $\beta_1$ is the $\ell_1$-Lipschitz continuity constant of the gradient, $\alpha_2$ is the $\ell_2$-quadratic growth constant, and $s$ is the sparsity of optimal solutions. However, these methods do not leverage the sparsity of optimal solutions towards also reducing the runtime of each iteration. In particular, for optimization over matrix nuclear norm balls, the existing methods rely  on non-Euclidean proximal subproblems which, to the best of our knowledge, do not admit known efficient implementations. In this work, we establish that linear convergence rates depending on this improved condition number can be obtained using only sparse updates, which may result in significantly improved overall running times. For example, for nuclear-norm-constrained optimization, our method requires only low-rank SVD computations on each iteration. Moreover, our algorithms are considerably simpler to implement.
\end{abstract} 


\section{Introduction}
Convex optimization techniques have been extremely successful in problems where the goal is to recover a sparse signal, e.g., a  vector with few nonzero entries or a low rank matrix, given certain measurements, for example, \cite{bickel2009simultaneous, candes2012exact, candes2009near, candes2011tight, candes2006stable, fazel2008compressed, juditsky2011accuracy, koltchinskii2011nuclear}.
Many such canonical convex models are captured by the following general and standard convex composite model:
\begin{align}\label{eq:optProb}
\min_{\x\in\E}\{F(\x) := f(\x) + \R(\x)\},
\end{align}
where $\E$ is a finite-dimensional Euclidean vector space equipped with an inner-product $\langle\cdot,\cdot\rangle$, $f:\E\rightarrow\reals$ is convex and differentiable with Lipschitz continuous gradient
, and $\R:\E\rightarrow\reals\cup\{+\infty\}$ is a proper, closed, and convex function. We let $F^*$ denote the optimal value of~\eqref{eq:optProb} and we let $\mX^*\subseteq\operatorname{dom}(\R)$ denote the set of minimizers, which we assume is nonempty.

While our results will not be limited to the following two specific scenarios, for the purpose of the introduction, and in the context of tasks which aim to recover entrywise-sparse vectors or low-rank matrices, it will be beneficial to consider the cases that the nonsmooth function $\R$ is either an indicator for an $\ell_1$-ball $\{\x\in\reals^n ~|~\Vert{\x}\Vert_1 \leq R\}$, or that it is an indicator for a matrix nuclear norm ball $\{\x\in\reals^{m\times n} ~|~\Vert{\mathbf{\sigma}(\x)}\Vert_1 \leq R\}$, where $\mathbf{\sigma}(\x)$ denotes the vector of singular values of $\x$. Indeed, in the first case $\R$ is a standard choice for promoting entrywise sparsity in optimal solutions, and the second plays an analogous role in promoting low matrix rank.
Throughout this work we consider Problem \eqref{eq:optProb} under the explicit assumption that all optimal solutions are indeed sparse, or at least nearly sparse in a sense that will be formally introduced in the sequel. Some upper bound on the level of sparsity will be assumed known, however running  multiple instances of our algorithms with various sparsity levels in parallel may circumvent this requirement. 

Before clearly presenting our motivation and contributions, we need to introduce some standard notation.
Let us denote by $\beta$ the Lipschitz continuity constant of the gradient of $f$ w.r.t. some norm $\Vert{\cdot}\Vert$, i.e., for all $\x,\y$ we have that $\Vert{\nabla{}f(\x) - \nabla{}f(\y)}\Vert_* \leq \beta\Vert{\x-\y}\Vert$, where $\Vert{\cdot}\Vert_*$ denotes the norm dual to $\Vert{\cdot}\Vert$. We assume also that $F$ satisfies a  quadratic growth condition \cite{necoara2019linear}, with constant $\alpha >0$ w.r.t. the same norm, i.e., that for any $\x\in\textrm{dom}(\R)$ it holds that $\min_{\x^*\in\mX^*}\Vert{\x-\x^*}\Vert^2 \leq \frac{2}{\alpha}\left({F(\x) - F(\x^*)}\right)$. 
In this setup, known proximal gradient methods can produce an $\epsilon$-approximated solution (in function value) in worst-case number of iterations that scales like $\frac{\beta}{\alpha}\log{1/\epsilon}$, or even $\sqrt{\frac{\beta}{\alpha}}\log{1/\epsilon}$ when accelerated methods are used \cite{nesterov2013introductory, beck2017first, necoara2019linear}. The ratio $\beta/\alpha$ is the \textit{condition number} of $f$ w.r.t. the norm $\Vert{\cdot}\Vert$ and is naturally dependent on the choice of norm. 

In the recent works  \cite{juditsky2023sparse, ilandarideva2024accelerated}, it was established that when $\R$ is an indicator for a convex and closed set, and all solutions are indeed sparse/low-rank, it is possible to design gradient methods that obtain linear convergence rates which depend on a \textit{mixed-norm condition number} of the form $\frac{\beta_1s}{\alpha_2}$, where $\beta_1$ is the Lipschitz continuity constant of the gradient w.r.t. the $\ell_1$ norm (applied to the entries in case of entrywise sparsity, and to the singular values in case of singular-value-wise sparsity), $\alpha_2$ is the  quadratic growth constant w.r.t. the $\ell_2$ norm, and $s$ is an upper-bound on the sparsity of optimal solutions. This condition number could indeed be far better than the pure $\ell_2$ or pure $\ell_1$ condition numbers (i.e., when both smoothness and quadratic growth are measured w.r.t. the same norm).
For instance, suppose that $\E=\reals^n$, $\R\equiv 0$, and that $f$ is convex quadratic of the following very simple form:
\begin{align}\label{eq:badFunc}
f(\x) = \frac{1}{2}(\x-\x^*)^{\top}\A(\x-\x^*); \quad \A := \I + \mathbf{1}\mathbf{1}^{\top},
\end{align}
where $\mathbf{1}$ denotes the all-ones vector in $\reals^n$.

Clearly, $\x^*$ is the unique optimal solution.
One can easily observe that the Euclidean condition number for this problem is given by $\frac{\beta_2}{\alpha_2} = \frac{\lambda_{\max}(\A)}{ \lambda_{\min}(\A)} = n+1$.
On the other hand, a simple calculation shows that the $\ell_1$-Lipschitz continuity constant of the gradient is given by $\beta_1 = \max_{i,j}\vert{\A_{i,j}}\vert = 2$. 
To bound the $\ell_1$-quadratic growth constant, we observe that due to the structure of the matrix $\A$, 
\begin{align*}
\forall \x: \quad  f(\x) - f(\x^*)  &\geq \frac{1}{2}\Vert{\x-\x^*}\Vert_2^2 \geq \frac{1}{2n}\Vert{\x-\x^*}\Vert_1^2.
\end{align*}
The last inequality may hold with equality, and so, the $\ell_1$-condition number $\frac{\beta_1}{\alpha_1}$ scales as $\Theta(n)$,  as in the Euclidean case.

However, in the example above, if $\x^*$ satisfies $\nnz(\x^*) \leq s \ll n$, where $\nnz(\cdot)$ denotes the number of nonzero entries, the above-mentioned sparsity-based mixed-norm condition number satisfies $\frac{\beta_1s}{\alpha_2} = 2s$, which is in particular independent of the ambient dimension, and thus may lead to far-better (worst-case) convergence rates for suitable first-order methods.

The example in Eq. \eqref{eq:badFunc}  could be directly extended to the matrix case by considering the space $\reals^{n\times n}$ and taking $f(\X) = \frac{1}{2}\Vert{\X-\X^*}\Vert_2^2 + \frac{1}{2}\left({\trace(\X-\X^*)}\right)^2$ with $\X^*$  satisfying $\rank(\X^*) \leq s \ll n$, and letting $\Vert{\cdot}\Vert_2$ here denote the Euclidean (Frobenius) norm, and instead of the entrywise $\ell_1$ norm, we consider the matrix nuclear norm ($\ell_1$ norm w.r.t. the vector of singular values).

The recent works \cite{juditsky2023sparse, ilandarideva2024accelerated}\footnote{While \cite{juditsky2023sparse, ilandarideva2024accelerated} consider a stochastic convex optimization setting in which the objective function $f$ is given as an expectation w.r.t. an unknown distribution from which we can sample, here our focus is on the standard deterministic first-order convex optimization model.} indeed gave first-order methods with convergence rates of the form $\tilde{O}\left({\frac{\beta_1s}{\alpha_2}\log\frac{1}{\epsilon}}\right)$ \cite{juditsky2023sparse} and even an accelerated rate of the form $\tilde{O}\left({\sqrt{\frac{\beta_1s}{\alpha_2}}\log\frac{1}{\epsilon}}\right)$ \cite{ilandarideva2024accelerated} (the $\tilde{O}(\cdot)$ notation hides certain poly-logarithmic factors), by applying a restarted non-Euclidean gradient method or its accelerated variant, together with hard-thresholding steps to sparsify intermediate iterates. In Section \ref{sec:acc} we give a simple derivation of this accelerated rate.

While these non-Euclidean gradient methods lead to superior convergence rates, they do not leverage the sparsity of optimal solutions towards additional computational gains, and in particular may still require prohibitive running times per iteration when the dimension is very high. For instance, in case solutions are entrywise sparse, these methods still apply dense updates which may  lead to prohibitive running time for updating the gradient direction from one iteration to the next. A more prominent example is the case that $\E$ is a space of real matrices and $\R$ is an indicator for a nuclear norm ball, in which case, to the best of our knowledge, these methods do not admit efficient implementations (even in moderate dimension) due to the complicated structure of the non-Euclidean proximal subproblem needed to be solved on each iteration (even when optimal solutions have low-rank), see more details in Section \ref{sec:acc} and  Remark \ref{remark:MDcomp}.


In this work we present a new and simple approach towards obtaining linear convergence rates that depend on the mixed-norm condition number $\frac{\beta_1s}{\alpha_2}$ (though our method only obtains rates that are linear in it and not accelerated rates that scale with its square root), while also leveraging the sparsity of optimal solutions towards more efficient iterations. The main novel aspect is that our algorithm only applies \textit{sparse updates} (with the same level of sparsity assumed for the optimal solutions). In particular, our algorithm updates its iterates similarly to the classical Frank-Wolfe algorithm \cite{jaggi2013revisiting}:
\begin{align}\label{eq:int:FW}
\x_{t+1} \gets (1-\gamma)\x_t + \gamma\v_t \quad \textrm{for some}\quad  \gamma\in[0,1],
\end{align}
where $\v_t\in\textrm{dom}(\R)$ has the same level of sparsity assumed for the optimal solutions of Problem \eqref{eq:optProb}. 

For instance, in case all optimal solutions have at most $s$ nonzero entries and $f:\reals^n\rightarrow \reals$ is a convex quadratic objective function of the form $f(\x) = \frac{1}{2}\x^{\top}\A\x + \b^{\top}\x$, the fact that $\v_t$ is $s$-sparse, reduces the time to update the gradient direction from $O(n^2)$ to $O(sn)$. In case $\E=\reals^{m\times n}$, $\R$ is an indicator function for a nuclear norm ball,  and all optimal solutions have rank at most $s \ll \min\{m,n\}$, the fact that $\rank(\v_t) \leq s$ implies that computing it will not require an expensive full-rank SVD computation, but only computing the $s$ leading components in the SVD, which is often far more efficient.

Two additional advantages which make our approach considerably simpler (compared to \cite{juditsky2023sparse, ilandarideva2024accelerated}) are: I. it does not require restarts, i.e., it is a single-loop algorithm, and II. the updates in our algorithm (the way $\v_t$ in Eq. \eqref{eq:int:FW} is computed) are Euclidean, which often leads to significantly simpler implementations.

Table \ref{table:res} provides a concise comparison  between our method and the accelerated method  \cite{ilandarideva2024accelerated} on the two sparse scenarios discussed above. For instance, we can see that in case of minimizing a convex quadratic function over an $\ell_1$ ball, our worst-case total runtime is better in the regime $s \ll \left({\sqrt{\frac{\alpha_2}{\beta_1}}}n\right)^{2/3}$.

\begin{table*}\renewcommand{\arraystretch}{1.7}
{\footnotesize
\begin{center}

\newcolumntype{C}[1]{>{\centering\arraybackslash}m{#1}}
\begin{tabular}{ | C{16.9em} | C{6em} | C{15em} |}
  \hline
  Algorithm & \#iterations to $\epsilon$ error & iteration runtime \\
  \hline
  \multicolumn{3}{|c|}{minimization over $\ell_1$-ball with $f(\x) = \frac{1}{2}\x^{\top}\A\x + \b^{\top}\x$, $\A\succeq 0$, $\nnz(\x^*)\leq s~ \forall \x^*\in\mX^*$ } \\
\hline
  Restarted Acc. Grad. with hard-thresholding \cite{ilandarideva2024accelerated} (see Section \ref{sec:acc})& $\sqrt{\frac{\beta_1s}{\alpha_2}}$ & $O(n^2)$\\ \hline
  Our Algorithm \ref{alg:weakProxFW} (see Corollary \ref{cor:L1})& $\frac{\beta_1s}{\alpha_2}$ & $O(sn)$   \\ \hline
  \multicolumn{3}{|c|}{minimization over matrix nuclear norm ball, $\rank(\X^*) \leq s~\forall\X^*\in\mX^*$} \\
\hline
  Restarted Acc. Grad. with hard-thresholding \cite{ilandarideva2024accelerated}  (see Section \ref{sec:acc})& $\sqrt{\frac{\beta_1s}{\alpha_2}}$ & no efficient method known (to the best of our knowledge)  \\ \hline
  Our Algorithm \ref{alg:weakProxFW} (see Corollary \ref{cor:NN})& $\frac{\beta_1s}{\alpha_2}$ & two low-rank SVDs of $m\times n$ matrix  (SVD rank = $s$)  \\ \hline

\end{tabular}\caption{Comparison of our method and a state-of-the-art first-order method with a linear accelerated convergence rate that scales with the mixed-norm condition number $\frac{\beta_1s}{\alpha_2}$ for optimization over an $\ell_1$ ball in $\reals^n$ or a matrix nuclear norm ball in $\reals^{m\times n}$ with sparse optimal solutions. The convergence rates are given only in terms of the dependence on the mixed-norm condition number, i.e., omitting logarithmic factors and universal constants.}\label{table:res}
\end{center}}
\vskip -0.2in
\end{table*}\renewcommand{\arraystretch}{1.5}

The rest of the paper is organized as follows. In Section \ref{sec:main} we present a unified setup and optimization problem which captures the sparse convex optimization problems we are interested in, a generic algorithm for solving this problem, and its convergence analysis. In Section \ref{sec:apps} we discuss concrete implementations of our generic algorithm and draw specific conclusions w.r.t. optimization over $\ell_1$  balls in $\reals^n$ or over nuclear norm balls in $\reals^{m\times n}$, justifying the result presented in Table \ref{table:res}. In Section \ref{sec:exp} we report simple numerical demonstrations. Finally, in Section \ref{sec:noAcc} we briefly discuss the difficulty of extending our approach towards obtaining methods that are both accelerated, in the sense that their convergence rate scales with the square root of the mixed-norm condition number, and rely only on sparse updates.

\section{Unified Setup, Algorithm, and Convergence Analysis}\label{sec:main}
Throughout this section we consider Problem \eqref{eq:optProb}, where we recall $f:\E\rightarrow\reals$ is convex and differentiable with Lipschitz continuous gradient with parameter $\beta>0$ w.r.t. some norm $\|\cdot\|$, $\|\x\|_2=\sqrt{\langle\x,\x\rangle}$ denotes the Euclidean norm, and we shall assume Problem \eqref{eq:optProb} satisfies a quadratic growth condition with parameter $\alpha_2>0$ w.r.t. the Euclidean norm, i.e.,
\begin{align}
\forall \x\in\operatorname{dom}(\R):\qquad
\min_{\x^*\in\mX^*}\|\x-\x^*\|_2^2
\leq \frac{2}{\alpha_2}\bigl(F(\x)-F^*\bigr).
\label{eq:qg}
\end{align}

The following assumption formalizes the notion of sparsity we use.
\begin{assumption}\label{ass:sparse}
There exists a closed (not necessarily convex) subset $\mE_s\subset \E$ and a scalar $s\geq 1$ such that $\operatorname{dom}(\R)\cap\mE_s \neq \emptyset$, and  for all $\x,\y\in\mE_s$,
$\|\x-\y\|\leq \sqrt{2s}\|\x-\y\|_2$.
\end{assumption} 

In the context of instances of Problem \eqref{eq:optProb} related to recovering sparse vectors with at most $s$ nonzero entries  or low-rank matrices with rank at most $s$, as discussed in the Introduction, we should indeed think of $\mE_s$ as the restriction of the space $\E$ to vectors with at most $s$ nonzero entries or to matrices with rank at most $s$, respectively.

Denote $\mS_s:=\operatorname{dom}(\R)\cap\mE_s$. A hard assumption that all optimal solutions to Problem \eqref{eq:optProb} have level of sparsity at most $s$ is formalized in our framework by assuming $\mX^*\subseteq\mE_s$. However, as mentioned in the Introduction, we will consider a more robust formulation which is captured by introducing the error quantity $\Delta_s$:
\begin{align}
\Delta_s
:=
\sup_{\x^*\in\mX^*}
\min_{\z\in\mS_s}
\left\{
F(\z)-F^*+\frac{\alpha_2}{4}\|\z-\x^*\|_2^2
\right\}.
\label{eq:sparse_approx_error}
\end{align}
That is, $\Delta_s$ measures the largest possible gap (in terms of the function $\z\mapsto F(\z)-F^*+\frac{\alpha_2}{4}\|\z-\x^*\|_2^2$) of any optimal solution $\x^*$ to the closest (in terms of this function) sparse point in $\mS_s$.
We shall assume $\Delta_s<\infty$. Note that the inner minimum in~\eqref{eq:sparse_approx_error} is indeed attained since for every fixed $\x^*\in\mX^*$, the extended-valued function
\[
\z\mapsto F(\z)-F^*+\frac{\alpha_2}{4}\|\z-\x^*\|_2^2
\]
is lower semicontinuous and coercive over the closed set $\mE_s$, and it is finite at some point since $\mS_s$ is nonempty. 

Importantly, in case $\mX^*\subseteq\mS_s$, i.e., all optimal solutions indeed have sparsity level at most $s$, we clearly have $\Delta_s=0$. Additionally, if there exists a unique optimal solution $\x^*\notin\mS_s$, using the quadratic growth, \eqref{eq:sparse_approx_error} could be simplified to $\Delta_s \leq \frac{3}{2}\min_{\z\in\mS_s}\{F(\z) - F^*\}$.

While not part of Assumption \ref{ass:sparse}, we shall implicitly assume that it is computationally efficient to compute a Euclidean projection onto $\mE_s$, and to compute the restricted composite proximal operation appearing in Algorithm~\ref{alg:weakProxFW}. In case one of these operations is not unique, we simply compute an arbitrary element in the set of minimizers. In Section \ref{sec:apps} we discuss concrete and efficient implementations in several cases of interest. 
We denote by $\Pi_{\mE_s}[\cdot]$ the Euclidean projection operation onto $\mE_s$. In the sparse vector or low-rank matrix settings $\Pi_{\mE_s}[\cdot]$ is simply hard-thresholding by retaining the largest entries or singular values.

We can now present our algorithm for solving Problem~\eqref{eq:optProb} (under the assumptions listed above). The iterates are updated by taking a convex combination between the current iterate and a new point $\v_t\in\mE_s$, while the latter is obtained from a restricted composite proximal operation centered at a hard-thresholded version of the current iterate. The gradient is evaluated at the current iterate $\x_t$, and the step-size is chosen either by a predefined rule or by exact line-search. 

The combination of solving a restricted proximal problem and aggregating its solution with the current point via a Frank-Wolfe-style convex combination is not new, see for instance \cite{allen2017linear, garber2019fast, garber2021improved}.  However, the combination of this approach with the sparsification of the current iterate $\widehat{\x}_t$, and a suitable analysis, is what allows us to obtain the novel mixed-norm-condition-number-dependent linear rates. 

The algorithm requires knowledge of the $\ell_2$ quadratic growth parameter $\alpha_2$. The predefined step size also requires knowledge of the smoothness parameter $\beta$, whereas the exact line-search option does not. In the sequel we discuss how to deal with the requirement of knowing $\alpha_2$. 
Note  that Algorithm \ref{alg:weakProxFW} requires an additional parameter $g$. This parameter can be tuned without any knowledge of the problem's parameters and it is introduced mainly for convenience, see Remark \ref{remark:g} following our main theorem.

\begin{algorithm}[t]
\caption{Composite method with sparse proximal oracle and hard-thresholding}
\label{alg:weakProxFW}
\begin{algorithmic}[1]
\State input: quadratic growth parameter $\alpha_2>0$, gradient coefficient $g>4$;
    \Statex \hspace{1.6em}for the predefined step size, also the smoothness parameter $\beta>0$
\State $\x_1\leftarrow$ some arbitrary point in $\operatorname{dom}(\R)$
\For{$t=1,2,\ldots$}
    \State $\widehat{\x}_t\leftarrow \Pi_{\mE_s}[\x_t]$
    \State $\displaystyle
    \v_t\in\argmin_{\v\in\mE_s}
    \left\{
    \langle \v-\widehat{\x}_t,\nabla f(\x_t)\rangle
    +\R(\v)
    +\frac{\alpha_2}{2g}\|\v-\widehat{\x}_t\|_2^2
    \right\}$
    \Statex \hspace{1.6em}equivalently,
    \Statex \hspace{1.6em}$\displaystyle
    \v_t\in\argmin_{\v\in\mE_s}
    \left\{
    \R(\v)+\frac{\alpha_2}{2g}
    \left\|\v-
    \left(\widehat{\x}_t-\frac{g}{\alpha_2}\nabla f(\x_t)\right)
    \right\|_2^2
    \right\}$
    \State choose either $\displaystyle
    \gamma_t=\frac{(\sqrt g-2)\alpha_2}
    {(\sqrt g-2)\alpha_2+2g^{3/2}\beta s}$ (predefined step size),
    \Statex \hspace{1.6em}or $\displaystyle
    \gamma_t\in\argmin_{\gamma\in[0,1]}
    F\bigl((1-\gamma)\x_t+\gamma\v_t\bigr)$ (exact line-search)
    \State $\x_{t+1}\leftarrow(1-\gamma_t)\x_t+\gamma_t\v_t$
\EndFor
\end{algorithmic}
\end{algorithm}

\begin{theorem}
\label{thm:main}
Consider the iterates $(\x_t)_{t\geq 1}$ of Algorithm~\ref{alg:weakProxFW}, using either step-size option, with any fixed $g>4$. Let $h_t:=F(\x_t)-F^*$ and define the constants
\begin{align*}
A_g := \frac{\sqrt g+2}{\sqrt g-2}
\max\left\{1,\frac{8}{\sqrt g-2}\right\}, \quad B_g :=\frac{(\sqrt g-2)^2\alpha_2}{(g-4)\alpha_2+2g^{3/2}(\sqrt g+2)\beta s}\in(0,1).
\end{align*}
For every $t\geq 1$, it holds that
\begin{align}
&\left(h_{t+1}-A_g\Delta_s\right)_+
\leq
\left(1-B_g\right)
\left(h_t-A_g\Delta_s\right)_+,
\label{eq:robust_linear_rate}
\end{align}
where $(a)_+:=\max\{a,0\}$. Consequently,  the algorithm finds a point satisfying
$F(\x_t)-F^*
\leq A_g\Delta_s+\epsilon$
after
\begin{align*}
&O\left(1+\frac{1}{B_g}\log_+\frac{F(\x_1)-F^*}{\epsilon}
\right) = O\left({1+\left({1+\frac{\beta s}{\alpha_2}}\right)\log_+\frac{F(\x_1)-F^*}{\epsilon}}\right) 
\end{align*}
iterations, where $\log_+(x) := \max\{0,\log{}x\}$ for all $x >0$ and $\log_+(0):=0$.
\end{theorem}
\begin{remark}[the parameter $g$]\label{remark:g}
Note  the parameter $g$ in the algorithm is tunable without any prior knowledge on the problem's data. Taking $g$ to be arbitrarily large pushes the coefficient $A_g$ associated with the approximation quantity $\Delta_s$ towards  the infimum $1$, however also pushes the contraction parameter $B_g$ towards $0$. Alternatively, in case $\Delta_s = 0$ and assuming knowledge of the quantities $\alpha_2, \beta$, one may attempt to use $g$ to maximize the contraction parameter $B_g$. From an empirical point of view, we observed  that taking $g$ to be small as possible (say $g=5$) works best while maintaining  the validity of theoretical convergence results.
\end{remark}
\begin{proof}[Proof of Theorem~\ref{thm:main}]
We prove the one-step estimate more generally at an arbitrary current iterate $\x_t\in\operatorname{dom}(\R)$ and when the proximal operation uses any $a\in(0,\alpha_2]$ in place of $\alpha_2$. Write $\v_t$ for the corresponding sparse point and $\x_{t+1}$ for its update using either step-size rule, with $a$ in place of $\alpha_2$ in the predefined step size. Set $\eta\in(0,1)$ such that
\begin{align}
\eta:=\frac{(\sqrt g-2)a}
{(\sqrt g-2)a+2g^{3/2}\beta s}\in(0,1),
\qquad
1-\eta=\frac{2g^{3/2}\beta s}
{(\sqrt g-2)a+2g^{3/2}\beta s},
\label{eq:analysis_eta}
\end{align}
and set $\p_t:=(1-\eta)\x_t+\eta\v_t$. The theorem will follow by taking $a=\alpha_2$. 
Fix any $\z\in\mS_s$. The identity
\[
\p_t=(1-\eta)\left(\x_t+\frac{\eta}{1-\eta}(\v_t-\z)\right)+\eta\z
\]
and the convexity and smoothness of $f$ give
\begin{align}
f(\p_t)
&\leq
(1-\eta)f\left(\x_t+\frac{\eta}{1-\eta}(\v_t-\z)\right)+\eta f(\z)
\notag\\
&\leq
(1-\eta)f(\x_t)+\eta f(\z)
+\eta\langle\v_t-\z,\nabla f(\x_t)\rangle
+\frac{\beta\eta^2}{2(1-\eta)}\|\v_t-\z\|^2.
\label{eq:convexity_smoothness_bridge}
\end{align}
Since $\v_t,\z\in\mE_s$, Assumption~\ref{ass:sparse} applies
to their difference. We also use the triangle inequality and
weighted Young's inequality: for $u,w\geq0$ and $\lambda>0$,
\[
(u+w)^2
=u^2+2uw+w^2
\leq (1+\lambda)u^2+(1+\lambda^{-1})w^2.
\]
Taking $u=\|\v_t-\widehat{\x}_t\|_2$,
$w=\|\z-\widehat{\x}_t\|_2$, and
$\lambda=2/(\sqrt g-2)>0$ (since $g>4$), we obtain
\begin{align}
\|\v_t-\z\|^2
&\leq 2s\|\v_t-\z\|_2^2
\leq 2s\left(
\|\v_t-\widehat{\x}_t\|_2
+\|\z-\widehat{\x}_t\|_2
\right)^2
\notag\\
&\leq
2s\left(1+\frac{2}{\sqrt g-2}\right)
\|\v_t-\widehat{\x}_t\|_2^2
+
2s\left(1+\frac{\sqrt g-2}{2}\right)
\|\z-\widehat{\x}_t\|_2^2
\notag\\
&=
\frac{2s\sqrt g}{\sqrt g-2}
\|\v_t-\widehat{\x}_t\|_2^2
+s\sqrt g\|\z-\widehat{\x}_t\|_2^2.
\label{eq:sparse_difference_bound}
\end{align}

After substituting~\eqref{eq:sparse_difference_bound}
into~\eqref{eq:convexity_smoothness_bridge}, the coefficients of
$\|\v_t-\widehat{\x}_t\|_2^2$ and
$\|\z-\widehat{\x}_t\|_2^2$ simplify, respectively, as
\[
\begin{aligned}
\frac{\beta\eta^2}{2(1-\eta)}
\frac{2s\sqrt g}{\sqrt g-2}
&=\frac{a\eta}{2g}, \qquad
\frac{\beta\eta^2}{2(1-\eta)}s\sqrt g
&=\frac{a\eta(\sqrt g-2)}{4g},
\end{aligned}
\]
where both equalities follow from~\eqref{eq:analysis_eta}.

Also, convexity of $\R$ gives
$\R(\p_t)\leq(1-\eta)\R(\x_t)+\eta\R(\v_t)$.
Combining these observations yields
\begin{align}
F(\p_t)
\leq {}& (1-\eta)F(\x_t)+\eta f(\z)
+\eta\langle\v_t-\z,\nabla f(\x_t)\rangle+\eta\R(\v_t)
\notag\\
&+\frac{a\eta}{2g}\|\v_t-\widehat{\x}_t\|_2^2
+\frac{a\eta(\sqrt g-2)}{4g}\|\z-\widehat{\x}_t\|_2^2.
\label{eq:fixed_candidate_bound}
\end{align}
By optimality of $\v_t$ in the proximal operation,
\begin{align*}
&\langle\v_t-\z,\nabla f(\x_t)\rangle+\R(\v_t)
+\frac{a}{2g}\|\v_t-\widehat{\x}_t\|_2^2
\leq\R(\z)+\frac{a}{2g}\|\z-\widehat{\x}_t\|_2^2.
\end{align*}
For either step-size option, $F(\x_{t+1})\leq F(\p_t)$. Thus, for every $\z\in\mS_s$,
\begin{align}
F(\x_{t+1})
\leq
(1-\eta)F(\x_t)+\eta F(\z)
+\frac{a\eta}{4\sqrt g}\|\z-\widehat{\x}_t\|_2^2.
\label{eq:compare_sparse_point}
\end{align}

Let $\x_t^*\in\argmin_{\x\in\mX^*}\|\x-\x_t\|_2$. By the definition of $\Delta_s$ in~\eqref{eq:sparse_approx_error}, there exists $\z_t\in\mS_s$ such that
\begin{align}
F(\z_t)-F^*
+\frac{\alpha_2}{4}\|\z_t-\x_t^*\|_2^2
\leq\Delta_s.
\label{eq:z_t}
\end{align}
Taking $\z=\z_t$ in~\eqref{eq:compare_sparse_point}, subtracting $F^*$, and using $a\leq\alpha_2$ gives
\begin{align}
h_{t+1}
\leq
(1-\eta)h_t
+\eta\bigl(F(\z_t)-F^*\bigr)
+\frac{\alpha_2\eta}{4\sqrt g}\|\z_t-\widehat{\x}_t\|_2^2.
\label{eq:robust_before_geometry}
\end{align}
Since $\widehat{\x}_t$ is a Euclidean projection of $\x_t$ onto $\mE_s$ and $\z_t\in\mE_s$,
\[
\|\x_t-\widehat{\x}_t\|_2
\leq\|\x_t-\z_t\|_2
\leq\|\x_t-\x_t^*\|_2+\|\z_t-\x_t^*\|_2.
\]
Hence,
\begin{align}
\|\z_t-\widehat{\x}_t\|_2
&\leq\|\z_t-\x_t\|_2+\|\x_t-\widehat{\x}_t\|_2
\leq 2\left(\|\x_t-\x_t^*\|_2+\|\z_t-\x_t^*\|_2\right),
\label{eq:robust_projection_bound}
\end{align}
and therefore
\begin{align}
\frac{\alpha_2}{4\sqrt g}\|\z_t-\widehat{\x}_t\|_2^2
\leq {}&
\frac{\alpha_2}{\sqrt g}\|\x_t-\x_t^*\|_2^2
+\frac{2\alpha_2}{\sqrt g}\|\x_t-\x_t^*\|_2\|\z_t-\x_t^*\|_2
\notag\\
&+\frac{\alpha_2}{\sqrt g}\|\z_t-\x_t^*\|_2^2.
\label{eq:robust_expand}
\end{align}
Quadratic growth gives
$\alpha_2\|\x_t-\x_t^*\|_2^2\leq2h_t$. Combining this with the weighted Young's inequality
$2uw\leq\lambda u^2+\lambda^{-1}w^2$, applied with
$u=\sqrt{2h_t}$, $w=\sqrt{\alpha_2}\|\z_t-\x_t^*\|_2$, $\lambda=\frac{\sqrt g-2}{\sqrt g+2}>0$, gives
\begin{align}
\frac{2\alpha_2}{\sqrt g}
\|\x_t-\x_t^*\|_2\|\z_t-\x_t^*\|_2
&\leq
\frac{2}{\sqrt g}\sqrt{2\alpha_2h_t}
\|\z_t-\x_t^*\|_2
\notag\\
&\leq
\frac{2\lambda}{\sqrt g}h_t
+\frac{\alpha_2}{\lambda\sqrt g}
\|\z_t-\x_t^*\|_2^2
\notag\\
&=
\frac{2(\sqrt g-2)}{\sqrt g(\sqrt g+2)}h_t
+
\frac{\alpha_2(\sqrt g+2)}{\sqrt g(\sqrt g-2)}
\|\z_t-\x_t^*\|_2^2.
\label{eq:cross_term}
\end{align}

Substituting~\eqref{eq:cross_term}
into~\eqref{eq:robust_expand} and using quadratic growth,
$\alpha_2\|\x_t-\x_t^*\|_2^2\leq2h_t$, gives
\begin{align*}
\frac{\alpha_2}{4\sqrt g}
\|\z_t-\widehat{\x}_t\|_2^2
&\leq
\left(
\frac{2}{\sqrt g}
+\frac{2(\sqrt g-2)}{\sqrt g(\sqrt g+2)}
\right)h_t +
\left(
\frac{\alpha_2}{\sqrt g}
+\frac{\alpha_2(\sqrt g+2)}{\sqrt g(\sqrt g-2)}
\right)\|\z_t-\x_t^*\|_2^2\\
&=
\frac{4}{\sqrt g+2}h_t
+\frac{2\alpha_2}{\sqrt g-2}\|\z_t-\x_t^*\|_2^2.
\end{align*}

Substituting this bound
into~\eqref{eq:robust_before_geometry} and collecting terms gives
\begin{align}
h_{t+1}
&\leq
\left(1-\eta+\frac{4\eta}{\sqrt g+2}\right)h_t
+\eta\bigl(F(\z_t)-F^*\bigr)
+
\frac{2\alpha_2\eta}{\sqrt g-2}
\|\z_t-\x_t^*\|_2^2
\notag\\
&=
\left(1-\frac{\sqrt g-2}{\sqrt g+2}\eta\right)h_t
+
\eta\left(
F(\z_t)-F^*
+\frac{2\alpha_2}{\sqrt g-2}\|\z_t-\x_t^*\|_2^2
\right).
\label{eq:robust_recurrence}
\end{align}

Since $F(\z_t)-F^*\geq0$, grouping the last two terms according to~\eqref{eq:z_t} implies
\begin{align}
h_{t+1}
\leq
\left(1-\frac{\sqrt g-2}{\sqrt g+2}\eta\right)h_t
+\eta\max\left\{1,\frac{8}{\sqrt g-2}\right\}\Delta_s.
\label{eq:affine_robust}
\end{align}
Subtracting the resulting error floor and taking positive parts gives
\begin{align}
&\left(h_{t+1}-\frac{\sqrt g+2}{\sqrt g-2}
\max\left\{1,\frac{8}{\sqrt g-2}\right\}\Delta_s\right)_+
\notag\\
&\quad\leq
\left(1-\frac{\sqrt g-2}{\sqrt g+2}\eta\right)
\left(h_t-\frac{\sqrt g+2}{\sqrt g-2}
\max\left\{1,\frac{8}{\sqrt g-2}\right\}\Delta_s\right)_+.
\label{eq:general_robust_linear_rate}
\end{align}
Thus, this estimate holds for either step-size option from any current iterate and every $a\in(0,\alpha_2]$, with $\eta$ as in~\eqref{eq:analysis_eta}.

Taking $a=\alpha_2$, plugging-in the value of $\eta$ and the constants $A_g, B_g$ defined in the theorem,  proves~\eqref{eq:robust_linear_rate}. Iterating~\eqref{eq:general_robust_linear_rate} gives
\begin{align*}
h_t
\leq {}&\frac{\sqrt g+2}{\sqrt g-2}
\max\left\{1,\frac{8}{\sqrt g-2}\right\}\Delta_s
\\
&+\left(1-\frac{\sqrt g-2}{\sqrt g+2}\eta\right)^{t-1}
\left(h_1-\frac{\sqrt g+2}{\sqrt g-2}
\max\left\{1,\frac{8}{\sqrt g-2}\right\}\Delta_s\right)_+.
\end{align*}
Using $1-u\leq e^{-u}$, it suffices to take
\[
t\geq1+\left\lceil
\frac{\sqrt g+2}{(\sqrt g-2)\eta}
\log_+\frac{h_1}{\epsilon}
\right\rceil.
\]
With $a=\alpha_2$ we have,
\[
\frac{\sqrt g+2}{(\sqrt g-2)\eta}
=\frac{\sqrt g+2}{\sqrt g-2}
+\frac{2g^{3/2}(\sqrt g+2)}{(\sqrt g-2)^2}\frac{\beta s}{\alpha_2} = \frac{1}{B_g},
\]
which proves the stated iteration complexity.
\end{proof}

The following corollary gives a straightforward approach towards circumventing exact knowledge of the quadratic growth parameter $\alpha_2$ by trying several candidates on each iteration. Note that the computations in \eqref{eq:multiple_trial_prox}-\eqref{eq:multiple_trial_candidate} can be executed in parallel. 

\begin{corollary}
\label{thm:multiple_trials}
Fix $g>4$ and let $a_1,\ldots,a_K>0$ be fixed trial values. Consider the modification of Algorithm~\ref{alg:weakProxFW} which, at each iteration, computes $\widehat{\x}_t$ as before and, for every $j\in\{1,\ldots,K\}$, computes
\begin{align}
\v_t^{(j)}
&\in\argmin_{\v\in\mE_s}
\left\{
\langle\v-\widehat{\x}_t,\nabla f(\x_t)\rangle
+\R(\v)+\frac{a_j}{2g}\|\v-\widehat{\x}_t\|_2^2
\right\},
\label{eq:multiple_trial_prox}\\
\gamma_t^{(j)}
&=\frac{(\sqrt g-2)a_j}{(\sqrt g-2)a_j+2g^{3/2}\beta s}
\quad\textrm{(predefined step size)},\notag\\
&\textrm{or}\quad\gamma_t^{(j)}\in\argmin_{\gamma\in[0,1]}
F\bigl((1-\gamma)\x_t+\gamma\v_t^{(j)}\bigr)
\quad\textrm{(exact line-search)},
\label{eq:multiple_trial_linesearch}\\
\widetilde{\x}_{t+1}^{(j)}
&:=(1-\gamma_t^{(j)})\x_t+\gamma_t^{(j)}\v_t^{(j)},
\label{eq:multiple_trial_candidate}
\end{align}
chooses $j_t\in\argmin_{j\in\{1,\ldots,K\}}F(\widetilde{\x}_{t+1}^{(j)})$ and sets
$\x_{t+1}:=\widetilde{\x}_{t+1}^{(j_t)}$. 
Suppose at least one trial value is at most $\alpha_2$ and define
$a_-:=\max\{a_j:a_j\leq\alpha_2\}$.
Then, with $h_t:=F(\x_t)-F^*$, for every $t\geq1$,
\begin{align}
&\left(h_{t+1}-A_g\Delta_s\right)_+
\leq
\left(1-\frac{(\sqrt g-2)^2a_-}
{(g-4)a_-+2g^{3/2}(\sqrt g+2)\beta s}\right)
\left(h_t-A_g\Delta_s\right)_+,
\label{eq:multiple_trial_linear_rate}
\end{align}
where $A_g$ is as defined in Theorem \ref{thm:main}.

Consequently, the same iteration complexity holds with $a_-$ in place of $\alpha_2$, namely
\[
O\left(
1+\left[
\frac{\sqrt g+2}{\sqrt g-2}
+\frac{2g^{3/2}(\sqrt g+2)}{(\sqrt g-2)^2}
\frac{\beta s}{a_-}
\right]
\log_+\frac{F(\x_1)-F^*}{\epsilon}
\right)
\]
iterations to reach
$F(\x_t)-F^*
\leq
A_g\Delta_s+\epsilon$. 
\end{corollary}

\begin{proof}
Let $j_-$ satisfy $a_{j_-}=a_-$. The proof of Theorem~\ref{thm:main} establishes~\eqref{eq:general_robust_linear_rate} for either step-size option from any current iterate and every $a\in(0,\alpha_2]$. Apply this estimate with $a=a_-$ and $\eta$ as in~\eqref{eq:analysis_eta} to $\widetilde{\x}_{t+1}^{(j_-)}$. Since
\[
F(\x_{t+1})
=\min_{1\leq j\leq K}F(\widetilde{\x}_{t+1}^{(j)})
\leq F(\widetilde{\x}_{t+1}^{(j_-)}),
\]
the same bound holds for the selected iterate $\x_{t+1}$. This proves~\eqref{eq:multiple_trial_linear_rate}. The complexity bound follows exactly as in Theorem~\ref{thm:main}.
\end{proof}

\section{Applications}\label{sec:apps}
In this section we discuss concrete instantiations  of our framework. Naturally, we focus on two main settings, one in which the sparsity of optimal solutions is in terms of the number of nonzero entries, and one in which it is in terms of matrix rank. In particular, in notable cases, when the function $\R$ satisfies certain symmetry conditions, either w.r.t. the entries of a vector input or the singular values / eigenvalues of a matrix input, the sparsity-constrained proximal problem in Algorithm \ref{alg:weakProxFW} admits efficient implementation.

\subsection{Entrywise-sparse optimal solutions}\label{sec:appsL1}
Consider Problem \eqref{eq:optProb} with the underlying space $\reals^n$ and when optimal solutions are entrywise-sparse (recall our results are somewhat robust to this assumption via the introduction of the quantity $\Delta_s$ in \eqref{eq:sparse_approx_error}). We instantiate the framework introduced in Section \ref{sec:main} by setting $\E = \reals^n$ and taking the sparse set in Assumption \ref{ass:sparse} to be $\mE_s = \{\x\in\reals^n ~|~\nnz(\x) \leq s\}$. The Euclidean projection $\Pi_{\mE_s}[\cdot]$ in Algorithm \ref{alg:weakProxFW} amounts in this case to zeroing-out all entries in the input vector, but  the $s$ largest-in-absolute-value (with ties broken arbitrarily). Algorithm \ref{alg:weakProxFW} requires an oracle for solving the restricted proximal problem:
\begin{align}\label{eq:subprob:L1}
\min_{\v\in\reals^n:~ \nnz(\v) \leq s}\R(\v) + \lambda\Vert{\v-\w}\Vert_2^2,
\end{align}
for any inputs $\lambda > 0, \w\in\reals^n$.

The following lemma is inspired by a result in \cite{beck2016minimization} but extends it to handle functions $\R$ beyond indicators for convex sets. A proof is given in the appendix.
\begin{lemma}
\label{lem:sparse_vector_prox}
Suppose $\operatorname{dom}(\R)\cap\{\x\in\reals^n~|~\nnz(\x) \leq s\} \neq \emptyset$ and  $\R$ is invariant to coordinate permutations, and either $\operatorname{dom}(\R)\subseteq\reals_+^n$ or $\R$ is invariant to coordinate-wise sign flips. Let $I=\{i_1,\ldots,i_s\}$ index the $s$ largest coordinates of $\w$ in the former case, or the $s$ largest-in-absolute-value coordinates of $\w$ in the latter case, with ties broken arbitrarily. Define
\[
\P_I:=[\e_{i_1},\ldots,\e_{i_s}],
\qquad
\w_I:=\P_I^\top\w,
\]
where $\e_i$ is the $i$th standard basis vector. Then, an optimal solution to Problem~\eqref{eq:subprob:L1} is $\v^*=\P_I\u^*$, where
\begin{align}
\u^*\in\argmin_{\u\in\reals^s}
\left\{\lambda\|\u-\w_I\|_2^2+\R(\P_I\u)\right\}.
\label{eq:sparse_vector_reduced_prox:1}
\end{align}
\end{lemma}

Lemma \ref{lem:sparse_vector_prox} in particular implies that when $\R$ is either an indicator for an $\ell_p$-ball, its restriction to the nonnegative orthant, or an $\ell_p$ regularizer,  the computation \eqref{eq:sparse_vector_reduced_prox} amounts to the computation of the standard proximal mapping associated with $\R$, but in the lower-dimensional space $\reals^s$, instead of $\reals^n$.

We now draw one specific runtime complexity result in order to demonstrate the overall appeal of our method in terms of worst-case total runtime.
\begin{corollary}[quadratic minimization over an $\ell_1$ ball / simplex]\label{cor:L1}
Suppose $\E=\reals^n$ and $\R$ is either an indicator for an $\ell_1$ ball or the unit simplex, and suppose $f$ is (convex) quadratic of the form $f(\x) = \frac{1}{2}\x^{\top}\A\x + \b^{\top}\x$. Suppose further that any optimal solution $\x^*$  satisfies $\nnz(\x^*) \leq s \ll n$.  Let the norm $\Vert{\cdot}\Vert$ introduced in Section \ref{sec:main} be the $\ell_1$ norm in $\reals^n$, and accordingly let $\beta_1=\max_{i,j}\vert{\A_{i,j}}\vert$ denote the $\ell_1$ Lipschitz continuity constant of $\nabla{}f$.
Algorithm \ref{alg:weakProxFW} admits an implementation that finds an $\epsilon$-approximated solution (in function value) in $O\left({\left({1+\frac{\beta_1s}{\alpha_2}}\right)\log\frac{F(\x_1)-F^*}{\epsilon}}\right)$ iterations, and each iteration takes $O(sn)$ runtime.
\end{corollary}
\begin{proof}
By Lemma~\ref{lem:sparse_vector_prox}, the restricted proximal step is obtained by selecting the $s$ largest coordinates (in absolute value for the $\ell_1$ ball) and projecting an $s$-dimensional vector onto the simplex or $\ell_1$ ball. These operations take $O(n+s\log s)=O(sn)$ time. Starting from a sparse feasible point, for the quadratic objective maintain $\nabla f(\x_t)=\A\x_t+\b$. Since $\v_t$ is $s$-sparse, computing $\A\v_t$ takes $O(sn)$ time, and both the exact line search and the update of the gradient direction from the previous one using
$\nabla f(\x_{t+1})=(1-\gamma_t)\nabla f(\x_t)+\gamma_t(\A\v_t+\b)$
require only $O(n)$ additional operations. The iteration bound follows from Theorem~\ref{thm:main} with $\Delta_s=0$.
\end{proof}

\subsection{Low-rank optimal solutions}\label{sec:appsNN}
Consider now Problem \eqref{eq:optProb} with the underlying space $\E=\reals^{m\times n}$ or $\E=\mbS^n$. In this case we take the sparse set in Assumption \ref{ass:sparse} to be $\mE_s = \{\x\in\E ~|~\rank(\x) \leq s\}$. The Euclidean projection $\Pi_{\mE_s}[\cdot]$ in Algorithm \ref{alg:weakProxFW} amounts in this case to zeroing-out all  singular values / eigenvalues in the input matrix, but the $s$ largest-in-absolute-value (with ties broken arbitrarily).
Algorithm \ref{alg:weakProxFW} requires an oracle for solving the following rank-restricted proximal problem:
\begin{align}\label{eq:subprob:NN}
\min_{\v\in\E: ~\rank(\v) \leq s}\R(\v) + \lambda\Vert{\v-\w}\Vert_2^2,
\end{align}
for any inputs $\lambda > 0, \w\in\E$.

In the following we set $d=n$ if $\E=\mbS^n$ and $d=\min\{m,n\}$ if $\E=\reals^{m\times n}$.

The following lemma is the matrix analogue to Lemma \ref{lem:sparse_vector_prox}. A proof is given in the appendix.
\begin{lemma}\label{lem:low_rank_prox}
Write $\mathbb{O}(k):=\{\Q\in\reals^{k\times k}:\Q^\top\Q=I_k\}$
and consider Problem~\eqref{eq:subprob:NN}.
For $\u\in\reals^s$, let $\operatorname{diag}(\u,\mathbf{0})\in\E$
denote the diagonal matrix obtained by padding $\u$ with zeros,
rectangular when appropriate.

If $\E=\reals^{m\times n}$, assume
$\R(\U\x\V^\top)=\R(\x)$ for all
$\x\in\E, \U\in\mathbb{O}(m), \V\in\mathbb{O}(n)$.
Let $\U_s,\V_s$ contain orthonormal left and right singular vectors
corresponding to the $s$ largest singular values of $\w$.
An optimal solution is
$\v^*=\U_s\operatorname{diag}(\u^*)\V_s^\top$, where
\begin{align}
\u^*\in\arg\min_{\u\in\reals^s}
\left\{
\lambda\sum_{i=1}^s(\u_i-\sigma_i(\w))^2
+\R\bigl(\operatorname{diag}(\u,\mathbf{0})\bigr)
\right\}.
\label{eq:low_rank_reduced_prox:1}
\end{align}

If $\E=\mbS^n$, assume
$\R(\Q\x\Q^\top)=\R(\x)$ for all
$\x\in\E, \Q\in\mathbb{O}(n)$,
and either $\operatorname{dom}(\R)\subseteq\mbS_+^n$
or $\R$ is invariant to eigenvalue sign-flipping.
Let $\theta_1,\ldots,\theta_s$ be the $s$ largest eigenvalues
of $\w$ in the former case, or the $s$ largest-in-absolute-value
eigenvalues in the latter, and let $\Q_s$ contain corresponding
orthonormal eigenvectors. An optimal solution is
$\v^*=\Q_s\operatorname{diag}(\u^*)\Q_s^\top$, where
\begin{align}
\u^*\in\arg\min_{\u\in\reals^s}
\left\{
\lambda\sum_{i=1}^s(\u_i-\theta_i)^2
+\R\bigl(\operatorname{diag}(\u,\mathbf{0})\bigr)
\right\}.
\label{eq:symmetric_low_rank_reduced_prox:1}
\end{align}
Ties may be broken arbitrarily.
\end{lemma}

Lemma \ref{lem:low_rank_prox} in particular implies that when $\R$ is either an indicator for a Schatten $p$ norm ball or regularizer (including $p=1$, i.e., the nuclear norm), or an indicator for the positive semidefinite cone,  Problem \eqref{eq:subprob:NN} is reduced to one low-rank SVD computation  (with rank equals $s$) or low-rank eigen-decomposition, and solving a simple convex problem in $\reals^s$.

\begin{corollary}[optimization with the nuclear norm  or over the  positive semidefinite cone]\label{cor:NN}
Suppose either $\E=\reals^{m\times n}$ and $\R$ is an indicator for a nuclear norm ball or a nuclear norm regularizer, or that $\E=\mbS^n$ and $\R$ is an indicator for the positive semidefinite cone. Suppose further that any optimal solution $\x^*$ satisfies $\rank(\x^*) \leq s$.  Let the norm $\Vert{\cdot}\Vert$ introduced in Section \ref{sec:main} be the nuclear norm (Schatten 1-norm), i.e., sum of singular values, and accordingly let $\beta_1$ denote the Lipschitz continuity constant of $\nabla{}f$ w.r.t. this norm. Algorithm \ref{alg:weakProxFW} admits an implementation that finds an $\epsilon$-approximated solution (in function value) in $O\left({\left(1+\frac{\beta_1s}{\alpha_2}\right)\log\frac{F(\x_1)-F^*}{\epsilon}}\right)$ iterations, and aside from the gradient computations, the runtime of each iteration is dominated by two rank-$s$ SVDs of $m\times n$ matrices (in case $\E = \reals^{m\times n}$) or two rank-$s$ eigen-decompositions of $n\times n$ matrices (in case $\E=\mbS^n$). 
\end{corollary}
\begin{proof}
For matrices of rank at most $s$, Assumption~\ref{ass:sparse} follows from
$\rank(\x-\y)\leq 2s$ and $\|\x-\y\|\leq\sqrt{2s}\|\x-\y\|_2$.
The hard-thresholding step requires the $s$ largest singular values (or, in the symmetric case, the $s$ largest-in-absolute-value eigenvalues). By Lemma~\ref{lem:low_rank_prox}, the restricted proximal step requires only another rank-$s$ SVD or eigen-decomposition, followed by a projection of the retained singular values onto an $\ell_1$ ball, soft-thresholding for the nuclear norm regularizer, or truncation of negative eigenvalues for the positive semidefinite cone. The iteration bound follows from Theorem~\ref{thm:main} with $\Delta_s=0$.
\end{proof}

\section{Numerical Evidence}\label{sec:exp}
We present simple numerical simulations for the two instances of Problem \eqref{eq:optProb} discussed in the Introduction: minimization of a convex quadratic function over an $\ell_1$ ball, and minimization of a convex quadratic function over a matrix nuclear norm ball. Both examples are based on the function \eqref{eq:badFunc} so that the Euclidean condition number scales with the ambient dimension, while the mixed-norm condition number $\beta_1s/\alpha_2$ scales only with the sparsity or rank of the optimal solution. The experiments focus on testing the empirical convergence rates (in terms of approximation error vs. number of iterations) of the considered methods.

We first consider the $\ell_1$-constrained problem
\begin{align}\label{eq:expL1}
\min_{\x\in\reals^n:\,\Vert{\x}\Vert_1\leq R}
\left\{f(\x):=\frac{1}{2}(\x-\x^*)^{\top}(\I+3\mathbf{1}\mathbf{1}^{\top})(\x-\x^*)\right\}.
\end{align}
As in the example in Eq. \eqref{eq:badFunc}, the $\ell_2$-quadratic growth (in fact, strong convexity) constant is $\alpha_2=1$, while the $\ell_2$ and $\ell_1$ Lipschitz continuity constants of the gradient are $\beta_2=1+3n$ and $\beta_1=4$, respectively. We set $n=1000$ and $R=10$, and consider $s=\nnz(\x^*)\in\{10,25,50,100\}$. For each instance, the support of $\x^*$ is drawn uniformly at random among all subsets of $s$ coordinates. Its nonzero entries are sampled independently from the standard Gaussian distribution and then rescaled so that $\Vert{\x^*}\Vert_1=R$. Thus, $\x^*$ is the unique minimizer of \eqref{eq:expL1}, with $f(\x^*)=0$.

Aside from our Algorithm \ref{alg:weakProxFW}, we implemented two benchmarks: I. a Euclidean accelerated proximal gradient method for smooth and strongly convex optimization, known as V-FISTA \cite{beck2017first}, which has worst-case convergence rate $O((1-\sqrt{\alpha_2/\beta_2})^t)$ (see Theorem 10.42 in \cite{beck2017first}), using the theoretical smoothness and strong convexity parameters; and II. the restarted non-Euclidean accelerated gradient method with hard-thresholding steps described in Section \ref{sec:acc}, which has a worst-case iteration complexity of $\tilde{O}\left(\sqrt{\frac{\beta_1s\ln(n)}{\alpha_2}}\log\frac{1}{\epsilon}\right)$. For the latter, we implemented Nesterov's Method 5.6 \cite{nesterov2005smooth}, see also Remark \ref{remark:accMD}. For the V-FISTA method we observed that even trying to slightly tune the prescribed theoretical step-size makes the method divergent, and so we stuck to the original implementation in \cite{beck2017first}. For the accelerated non-Euclidean method we set the length of each epoch (the parameter $K$ in Algorithm \ref{alg:restartAccGD}) to $K = \left\lceil{25 \sqrt{\frac{\beta_1s}{\alpha_2}}}\right\rceil$ which gave the best results. For this method, we report the approximation errors at the ends of the epochs against the cumulative number of inner iterations.
For our Algorithm \ref{alg:weakProxFW}, we implemented two variants, both using exact line-search to determine $\gamma_t$ in each iteration, but differing in the value of the proximal coefficient $g$ in Algorithm \ref{alg:weakProxFW}: I. the choice $g=5$, which satisfies the requirement $g>4$ in Theorem \ref{thm:main}, and II. the heuristic choice $g=1$, which uses a larger proximal coefficient but is not covered by the theorem. For simplicity, in all algorithms that require a sparsity parameter we use the true value $s=\nnz(\x^*)$.

We initialize all algorithms at zero and average the approximation errors over $10$ independently generated problem instances (each drawing a new $\x^*$). Figure \ref{fig:expL1} displays the natural logarithm of the average approximation error $f(\x_t)-f^*$   versus the iteration number.

We next consider the matrix analogue of \eqref{eq:expL1}:
\begin{align}\label{eq:expNN}
\min_{\X\in\reals^{n\times n}:\,\Vert{\X}\Vert_{nuc}\leq R}
\left\{f(\X):=\frac{1}{2}\Vert{\X-\X^*}\Vert_F^2+
\frac{3}{2}\bigl(\trace(\X-\X^*)\bigr)^2\right\},
\end{align}
where $\Vert{\X}\Vert_{nuc}=\Vert{\sigma(\X)}\Vert_1$ denotes the matrix nuclear norm and $\Vert{\cdot}\Vert_F$ denotes the Frobenius norm. Again, the Frobenius quadratic growth constant is $\alpha_2=1$, while the smoothness constants w.r.t. the Frobenius and nuclear norms are $\beta_2=1+3n$ and $\beta_1=4$, respectively. Here we set $n=500$, $R=10$, and $r=\rank(\X^*)\in\{5,10,20,50\}$. For each instance, we generate two independent $n\times r$ matrices with i.i.d. standard Gaussian entries and orthonormalize their columns to obtain $\U,\V\in\reals^{n\times r}$. We then set $\X^*=\U\operatorname{diag}(\sigma)\V^{\top}$, where the singular values are generated as $\sigma_i=R|z_i|/\sum_{j=1}^r|z_j|$ for independent standard Gaussian variables $z_1,\ldots,z_r$. Thus, $\rank(\X^*)=r$ and $\Vert{\X^*}\Vert_{nuc}=R$, and $\X^*$ is the unique minimizer with $f(\X^*)=0$.

For the matrix experiments, we compare our two variants of Algorithm \ref{alg:weakProxFW} ($g=5$ and $g=1$) against V-FISTA, using the same choices of parameters and exact line-search for our algorithm as in the vector case. 
We do not include the restarted non-Euclidean accelerated method in this comparison, since, as discussed in the Introduction, its non-Euclidean proximal subproblems over the nuclear norm ball do not admit an efficient method (see also Remark \ref{remark:MDcomp} in the appendix). We set the rank parameter of our algorithm to the true rank $r=\rank(\X^*)$.

Here also we initialize all algorithms at zero and average the approximation errors over $10$ independently generated problem instances (each drawing a new $\X^*$). Figure \ref{fig:expNN} displays the average approximation error $f(\X_t)-f^*$.

Let us comment on the results displayed in Figure \ref{fig:expL1}. The heuristic choice $g=1$ for our Algorithm \ref{alg:weakProxFW} clearly outperforms all other methods in all four sparsity settings. The theoretically justified choice $g=5$ also performs very well when $s\in\{10,25,50\}$, but deteriorates as the sparsity increases; in particular, for $s=100$ it is eventually slower than V-FISTA. Somewhat surprisingly, the restarted non-Euclidean accelerated gradient method performs similarly to V-FISTA when $s=10$, but is slower for the larger sparsity levels considered, despite its theoretically superior sparsity-dependent worst-case convergence guarantee.

\begin{figure}[htbp]
\centering
\begin{subfigure}{0.41\textwidth}
    \centering
    \includegraphics[width=\linewidth]{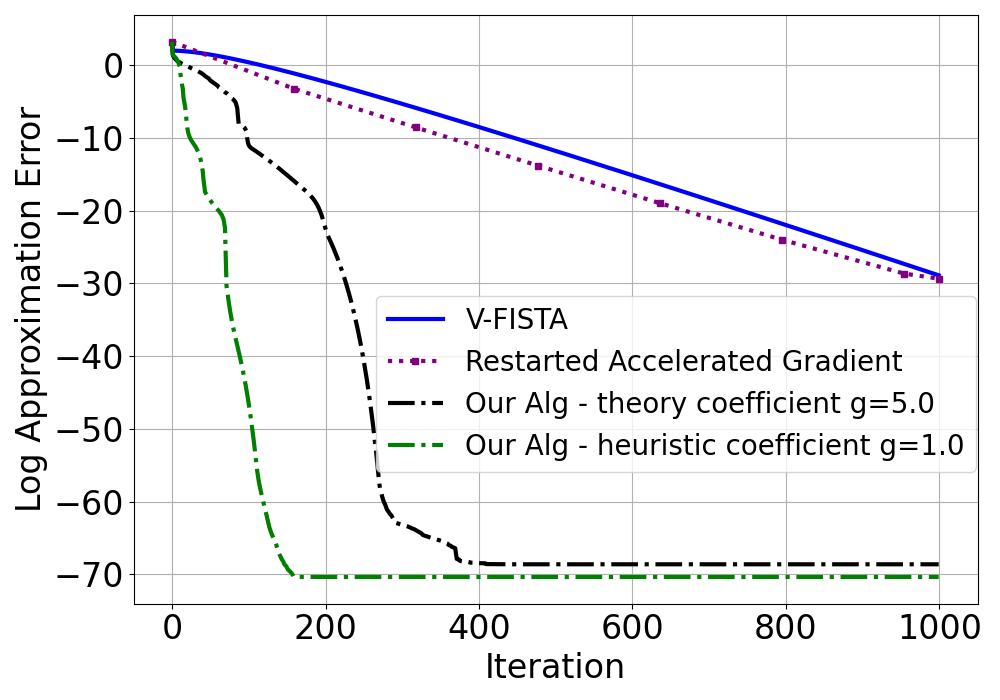}
    \caption{$\nnz(\x^*)=10$}
\end{subfigure}\hfill
\begin{subfigure}{0.41\textwidth}
    \centering
    \includegraphics[width=\linewidth]{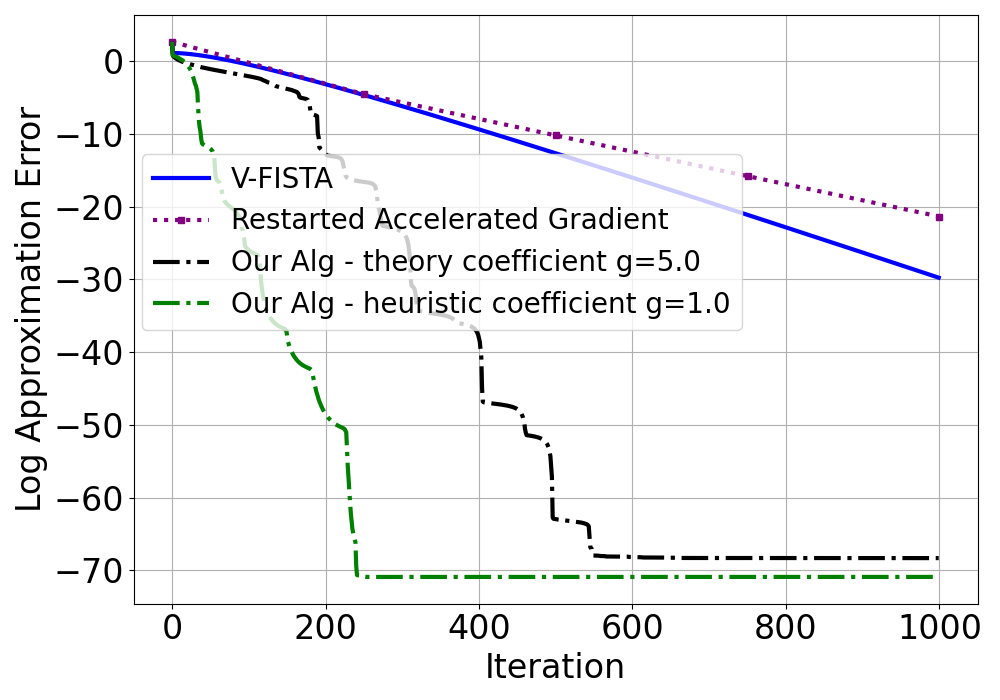}
    \caption{$\nnz(\x^*)=25$}
\end{subfigure}

\medskip
\begin{subfigure}{0.41\textwidth}
    \centering
    \includegraphics[width=\linewidth]{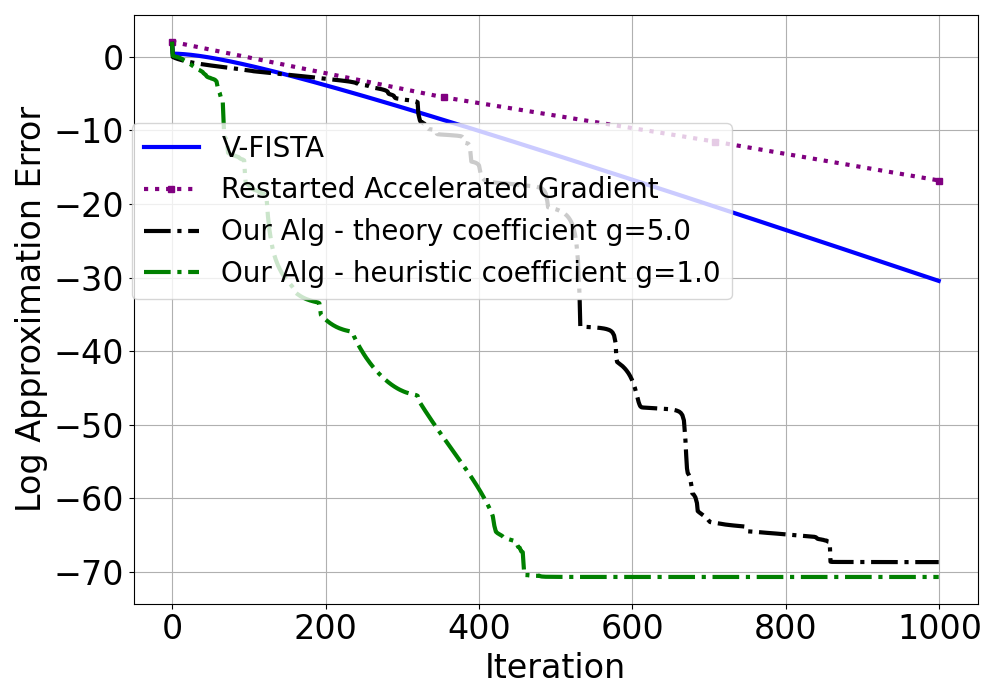}
    \caption{$\nnz(\x^*)=50$}
\end{subfigure}\hfill
\begin{subfigure}{0.41\textwidth}
    \centering
    \includegraphics[width=\linewidth]{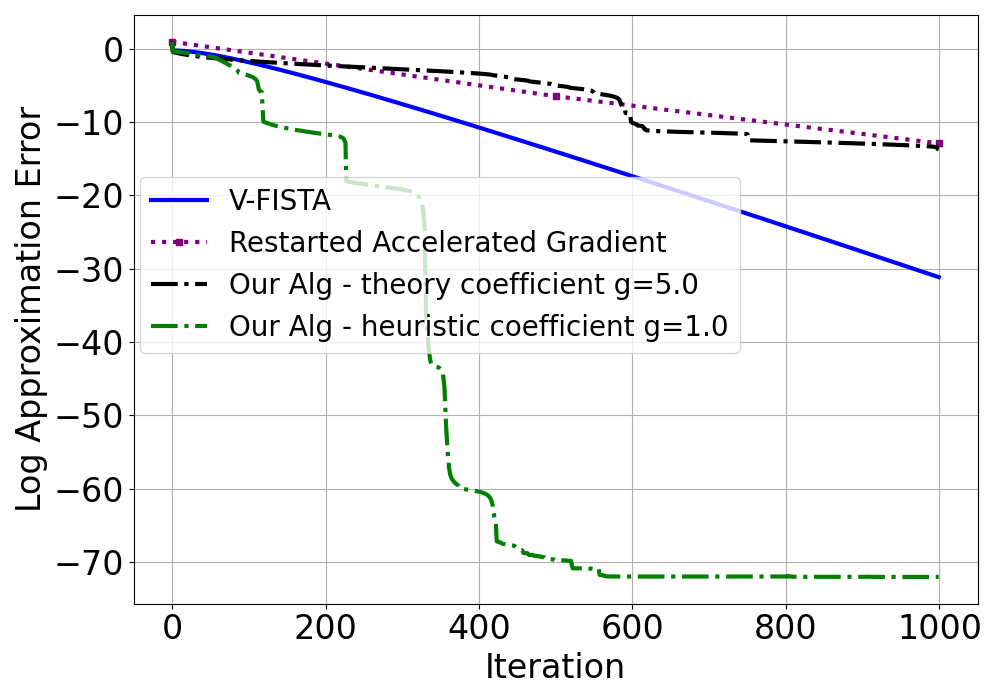}
    \caption{$\nnz(\x^*)=100$}
\end{subfigure}
\caption{Empirical convergence of the algorithms for the $\ell_1$-constrained problem \eqref{eq:expL1}.} 
\label{fig:exp}\label{fig:expL1}
\end{figure}

We now comment on the results displayed in Figure \ref{fig:expNN}. As in the $\ell_1$ experiments, our algorithm with the heuristic choice $g=1$ clearly outperforms V-FISTA and the variant with $g=5$ for all ranks considered. The theoretically justified choice $g=5$ is also competitive for lower ranks, in particular for $r=5$, but its performance deteriorates as the rank increases. Importantly, the plots compare convergence in terms of iteration counts rather than running times. 

\begin{figure}[htbp]
\centering
\begin{subfigure}{0.41\textwidth}
    \centering
    \includegraphics[width=\linewidth]{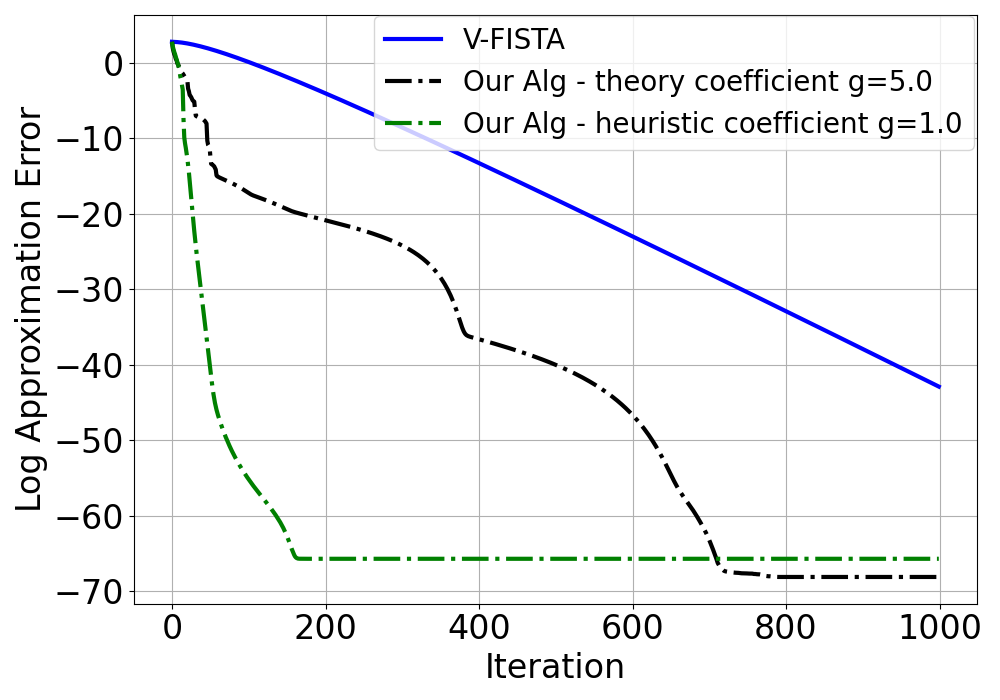}
    \caption{$\rank(\X^*)=5$}
\end{subfigure}\hfill
\begin{subfigure}{0.41\textwidth}
    \centering
    \includegraphics[width=\linewidth]{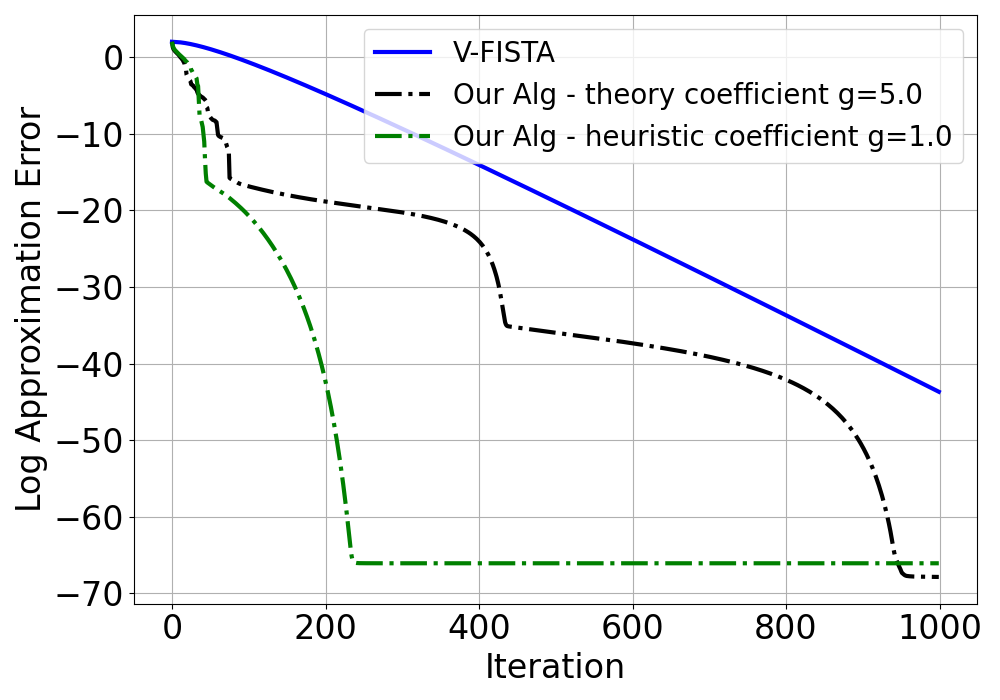}
    \caption{$\rank(\X^*)=10$}
\end{subfigure}

\medskip
\begin{subfigure}{0.41\textwidth}
    \centering
    \includegraphics[width=\linewidth]{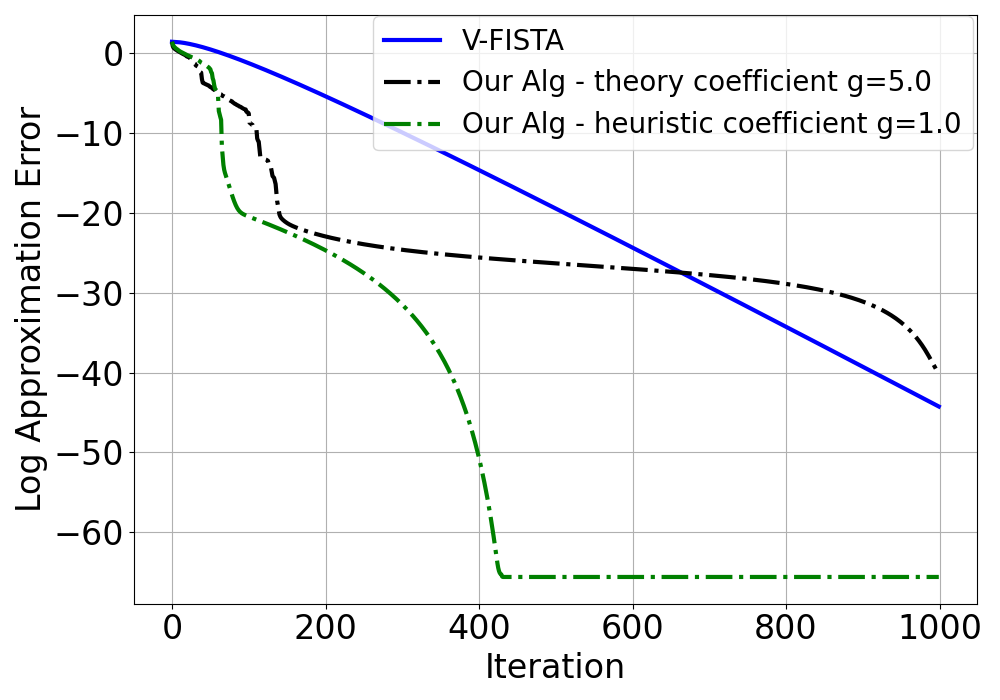}
    \caption{$\rank(\X^*)=20$}
\end{subfigure}\hfill
\begin{subfigure}{0.41\textwidth}
    \centering
    \includegraphics[width=\linewidth]{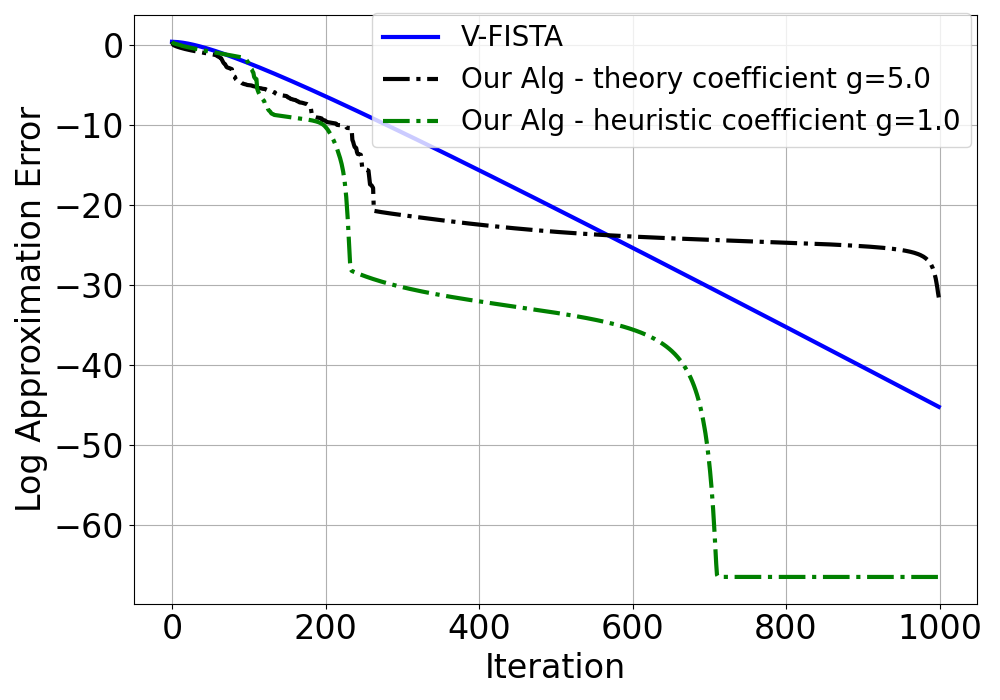}
    \caption{$\rank(\X^*)=50$}
\end{subfigure}
\caption{Empirical convergence of the algorithms for the nuclear-norm-constrained problem \eqref{eq:expNN}. }
\label{fig:expNN}
\end{figure}

\FloatBarrier

\section{On Acceleration with Sparse Updates}\label{sec:noAcc}
A natural question is whether Algorithm \ref{alg:weakProxFW} can be accelerated, achieving a linear convergence rate with square-root dependence on \(\beta_1 s/\alpha_2\), while retaining the sparsity restriction \(\v_t \in \mathcal E_s\) in the proximal subproblem. Unfortunately, the lower bounds in our recent  work \cite{garber2026complexities} provide a negative answer, establishing that, in general, first-order methods relying on such sparsity-constrained proximal operations cannot achieve accelerated convergence rates.

\section*{Acknowledgments} 
This work was funded by the European Union (ERC, ProFreeOpt, 101170791). Views and opinions expressed are however those of the author(s) only and do not necessarily reflect those of the European Union or the European Research Council Executive Agency. Neither the European Union nor the granting authority can be held responsible for them.

We would like to thank Amir Beck and Marc Teboulle for fruitful discussions during the early stages of this work.
\appendix

\section{Accelerated Rates with Dense Updates --- Recovering the rates in \cite{ilandarideva2024accelerated}}\label{sec:acc}
In this section we give a concise description of the accelerated linear convergence rates reported in \cite{ilandarideva2024accelerated}, which scale only with $\sqrt{\frac{\beta_1{}s}{\alpha_2}}$, but which rely on \textit{non-sparse} updates, and in fact may lead to subproblems that do not admit (to the best of our knowledge) efficient methods. While the methods in \cite{ilandarideva2024accelerated} were developed for stochastic optimization, here we are concerned with a purely deterministic setting. For ease of presentation, throughout this section we will focus on the case that $\R$ is an indicator function for an $\ell_1$-ball in $\reals^n$: $\{\x\in\reals^n~|~\Vert{\x}\Vert_1\leq R\}$ for some $R >0$. The extension to our second example of interest, a nuclear norm ball in $\reals^{m\times n}$ follows in a straightforward manner by replacing the entrywise $\ell_1$ norm with the matrix nuclear norm, i.e., the $\ell_1$ norm w.r.t. the vector of singular values. The method is a straightforward combination of a restarted-accelerated non-Euclidean proximal gradient method due to Nesterov \cite{nesterov2005smooth} with sparse projection steps of the iterates.

Throughout this section we let $\ball_1(R)$ denote the origin-centered $\ell_1$-ball of radius $R$ in $\reals^n$. For ease of presentation, here (and only here) we assume that Problem \eqref{eq:optProb} admits a unique optimal solution $\x^*$  and that $\nnz(\x^*) \leq s \ll n$. 

The following result is well-known, see for instance \cite{ilandarideva2024accelerated} .
\begin{lemma}\label{lem:scnorm}
Given $\x_0\in\ball_1(R)$ and assuming $n \geq 3$, consider the prox function 
\begin{align}\label{eq:scnrom:1}
d_{\x_0}(\x) := \frac{1}{2}e\log{n}\cdot{}n^{(p-1)(2-p)/p}\Vert{\x-\x_0}\Vert_p^2 \quad  \textrm{for} \quad p = 1 + 1/\log{n}. 
\end{align}
Then, $d_{\x_0}(\x)$ is $1$-strongly convex over $\ball_1(R)$ w.r.t. the norm $\Vert{\cdot}\Vert_1$, and for any $\y\in\ball_1(R)$, $d_{\x_0}(\y) \leq \frac{e^2\log{n}}{2}\Vert{\y-\x_0}\Vert_1^2$.
\end{lemma}

\begin{algorithm}
\caption{Nesterov's accelerated gradient method (see Method 3.11 in  \cite{nesterov2005smooth})}
\label{alg:accGD}
\begin{algorithmic}[1]
\State input: initialization point $\x_0\in\ball_1(R)$, $\ell_1$-smoothness parameter $\beta_1$
\State define the function $d(\z)$ as in Eq. \eqref{eq:scnrom:1} and set $\sigma = 1$
\For{$k=0,1\dots $}
\State $\y_k \gets  \arg\min_{\y\in\ball_1(R)}\langle{\y-\x_k, \nabla{}f(\x_k)}\rangle + \frac{\beta_1}{2}\Vert{\y-\x_k}\Vert_1^2$
\State $\z_k \gets \arg\min_{\z\in\ball_1(R)}\left\{{\frac{\beta_1}{\sigma}d_{\x_0}(\z) + \sum_{i=0}^k\frac{i+1}{2}\langle{\z-\x_i, \nabla{}f(\x_i)}\rangle}\right\}$ 
\State $\x_{k+1} \gets \frac{2}{k+3}\z_k + \frac{k+1}{k+3}\y_k$
\EndFor
\end{algorithmic}
\end{algorithm}

\begin{theorem}[Theorem 2 in \cite{nesterov2005smooth}]\label{thm:accGrad}
When applied to Problem \eqref{eq:optProb} with $\R$ being the indicator function for $\ball_1(R)$, 
the sequence $(\y_k)_{k \geq 0}$ produced by Algorithm \ref{alg:accGD} satisfies
\begin{align*}
\forall k\geq 0:\qquad f(\y_k) - f(\x^*) \leq \frac{4\beta_1d_{\x_0}(\x^*)}{(k+1)(k+2)}.
\end{align*}
\end{theorem}

\begin{algorithm}
\caption{Restarted accelerated gradient method}\label{alg:restartAccGD}
\begin{algorithmic}[1]
\State input: sparsity parameter $s\in[n]$, epoch length $K$
\State $\x_0 \gets$ some point in $\ball_1(R)$ such that $\nnz(\x_0) \leq s$
\For{$t=0,1\dots $}
\State $\x_{t+1} \gets$ the point $\y_K$ computed by Algorithm \ref{alg:accGD} after $K$ iterations, when initialized with the point $\widehat{\x}_t = \arg\min_{\x\in\reals^n:~\nnz(\x) \leq s}\Vert{\x-\x_t}\Vert_2$ \Comment{i.e., $\widehat{\x}_t$ equals $\x_t$ on the $s$ largest (in absolute value) entries, and zero elsewhere}
\EndFor
\end{algorithmic}
\end{algorithm} 

\begin{theorem}
Consider applying Algorithm \ref{alg:restartAccGD} to Problem \eqref{eq:optProb} with $\R$ being an indicator function for $\ball_1(R)$ and  with $K \geq \sqrt{\frac{64s\beta_1e^2\ln(n)}{\alpha_2}}-1$. Then,
\begin{align*}
\forall t\geq 0: \qquad f(\x_{t+1}) - f(\x^*) \leq \frac{\alpha_2R^2}{s}2^{-(t+3)}.
\end{align*}
\end{theorem}
Thus, to reach an $\epsilon$-approximated solution (in function value) the number of calls to the first-order oracle of $f$ is $O\left({\sqrt{\frac{\beta_1{}s\ln(n)}{\alpha_2}}\ln\left({\frac{\alpha_2R^2}{s\epsilon}}\right)}\right)$.
\begin{proof}
Fix some iteration $t \geq 0$ of Algorithm \ref{alg:restartAccGD}. Let $d_t(\x)$ denote the function $d_{\x_0}(\x)$, as defined in Eq. \eqref{eq:scnrom:1}, with $\x_0 = \widehat{\x}_t$ --- the initialization point in iteration $t$ of  Algorithm \ref{alg:restartAccGD}. It holds that,
\begin{align*}
\Vert{\x^*-\widehat{\x}_{t+1}}\Vert_1^2 &\underset{\mathrm{(a)}}{\leq} 2s\Vert{\x^*-\widehat{\x}_{t+1}}\Vert_2^2 \\
&\leq 4s\Vert{\x^*-\x_{t+1}}\Vert_2^2 + 4s\Vert{\x_{t+1}-\widehat{\x}_{t+1}}\Vert_2^2 \\
&\underset{\mathrm{(b)}}{\leq}  8s\Vert{\x^*-\x_{t+1}}\Vert_2^2 \underset{\mathrm{(c)}}{\leq} \frac{16s}{\alpha_2}\left({f(\x_{t+1}) - f(\x^*)}\right) \\
&\underset{\mathrm{(d)}}{\leq} \frac{64s\beta_1d_t(\x^*)}{\alpha_2(K+1)^2} \underset{\mathrm{(e)}}{\leq} \frac{32s\beta_1e^2\ln(n)\cdot{}\Vert{\x^*-\widehat{\x}_t}\Vert_1^2}{\alpha_2(K+1)^2},
\end{align*}
where $\mathrm{(a)}$ holds since $\x^*$ is $s$-sparse by assumption and $\widehat{\x}_{t+1}$ is $s$-sparse by construction, $\mathrm{(b)}$ holds by definition of $\widehat{\x}_{t+1}$, $\mathrm{(c)}$ holds due to the $\ell_2$-quadratic growth assumption, $\mathrm{(d)}$ follows from Theorem \ref{thm:accGrad} (when applied with the initialization point $\widehat{\x}_t$ and as a consequence the prox function is $d_t(\cdot)$ as defined above), and $\mathrm{(e)}$ follows from upper-bounding $d_t(\x^*)$ using Lemma \ref{lem:scnorm}. 

Thus, for $K \geq \sqrt{\frac{64s\beta_1e^2\log(n)}{\alpha_2}}-1$, and noting that $\widehat{\x}_0 = \x_0$, we can deduce that
\begin{align*}
\Vert{\x^*-\widehat{\x}_t}\Vert_1^2 \leq 2^{-t}\Vert{\x^*-\x_0}\Vert_1^2.
\end{align*}
Using Theorem \ref{thm:accGrad} and Lemma \ref{lem:scnorm} again, this yields
\begin{align*}
f(\x_{t+1}) - f(\x^*) &\leq \frac{4\beta_1d_t(\x^*)}{(K+1)^2} \leq \frac{2\beta_1e^2\log(n)\Vert{\x^*-\widehat{\x}_t}\Vert_1^2}{(K+1)^2}  \\
&\leq \frac{\alpha_2}{s}2^{-(t+5)}\Vert{\x^*-\x_0}\Vert_1^2 \leq \frac{\alpha_2R^2}{s}2^{-(t+3)}.
\end{align*}
\end{proof}

\begin{remark}\label{remark:accMD}
In the derivations above, for the sake of being concrete, we chose to present Nesterov's Method 3.11 from  \cite{nesterov2005smooth} as Algorithm \ref{alg:accGD}. A different choice could have been Method 5.6 from the same paper for which Theorem \ref{thm:accGrad} still holds, however, the practical implementation of Method 5.6 might be easier since it requires that both prox computations are w.r.t. the same function $d_{\x_0}(\cdot)$, instead of one w.r.t. the $\ell_1$ norm and one w.r.t. the function $d_{\x_0}(\cdot)$ as in Algorithm \ref{alg:accGD}.
\end{remark}

\begin{remark}[implementing Algorithm \ref{alg:accGD}]\label{remark:MDcomp}
Algorithm \ref{alg:accGD} requires on each iteration (among other operations) solving a non-Euclidean proximal problem over the domain (the computation of the point $\z_k$). To the best of our knowledge, both in the case that the domain is an $\ell_1$ ball in $\reals^n$ and in the case that it is a nuclear norm ball in $\reals^{m\times n}$, this computation does not admit a simple implementation.  For the $\ell_1$ ball, a procedure based on bisection which has complexity $\tilde{O}(n)$, where the $\tilde{O}(\cdot)$ notation hides logarithmic factors, is known. 
However, in the nuclear-norm-ball setting, the proximal subproblem involves a Schatten-$p$ regularizer centered at the restart point, while the feasible nuclear-norm ball remains centered at the origin. Consequently, the usual reduction of unshifted spectral proximal problems to optimization over singular values does not directly apply. We are not aware of an implementation of this translated constrained proximal problem whose cost is comparable to a constant number of SVDs. One approach, which is highly inefficient  is to solve this subproblem via an iterative first-order method, however, this will potentially require a full SVD computation (in order to project onto the nuclear norm ball or to evaluate the gradient of the Schatten-$p$
regularizer) on each iteration of this inner solver.
\end{remark}

\section{Proofs of Lemma \ref{lem:sparse_vector_prox} and Lemma \ref{lem:low_rank_prox}}

We first restate Lemma \ref{lem:sparse_vector_prox} and then prove it.
\begin{lemma}
Suppose $\operatorname{dom}(\R)\cap\{\x\in\reals^n~|~\nnz(\x) \leq s\} \neq \emptyset$ and  $\R$ is invariant to coordinate permutations, and either $\operatorname{dom}(\R)\subseteq\reals_+^n$ or $\R$ is invariant to coordinate-wise sign flips. Let $I=\{i_1,\ldots,i_s\}$ index the $s$ largest coordinates of $\w$ in the former case, or the $s$ largest-in-absolute-value coordinates of $\w$ in the latter case, with ties broken arbitrarily. Define
\[
\P_I:=[\e_{i_1},\ldots,\e_{i_s}],
\qquad
\w_I:=\P_I^\top\w,
\]
where $\e_i$ is the $i$th standard basis vector. Then, an optimal solution to Problem~\eqref{eq:subprob:L1} is $\v^*=\P_I\u^*$, where
\begin{align}
\u^*\in\argmin_{\u\in\reals^s}
\left\{\lambda\|\u-\w_I\|_2^2+\R(\P_I\u)\right\}.
\label{eq:sparse_vector_reduced_prox}
\end{align}
\end{lemma}

\begin{proof}
Fix any $\v\in\operatorname{dom}(\R)\cap\mE_s$. If $\supp(\v)\not\subseteq I$, there exist $i\notin I$ and $j\in I$ such that $\v_i\neq 0$ and $\v_j=0$.

If $\operatorname{dom}(\R)\subseteq\reals_+^n$, then $\v_i>0$ and $\w_j\geq \w_i$. Swapping coordinates $i$ and $j$ gives a point $\widetilde{\v}$ with the same value of $\R$ and
\[
\|\widetilde{\v}-\w\|_2^2-\|\v-\w\|_2^2
=2\v_i(\w_i-\w_j) + 2\v_j(\w_j-\w_i)\leq 0,
\]
where we recall that $\v_j = 0$.

If instead $\R$ is invariant to coordinate-wise sign flips, move $|\v_i|$ to coordinate $j$, with a sign agreeing with $\w_j$ and an arbitrary sign if $w_j=0$, and set coordinate $i$ to zero. The resulting point $\widetilde{\v}$ again has the same value of $\R$, and
\[
\|\widetilde{\v}-\w\|_2^2-\|\v-\w\|_2^2
=2\v_i\w_i-2|\v_i||\w_j|\leq 0,
\]
since $j\in I, i\notin I$ implies that $|\w_j|\geq |\w_i|$.

Repeating this operation gives a point fully supported on $I$ with no larger objective value. 

For $\v=\P_I\u$, the objective in~\eqref{eq:subprob:L1} equals
\[
\lambda\|\u-\w_I\|_2^2+\R(\P_I\u)
+\lambda\sum_{i\notin I}\w_i^2,
\]
and thus Problem \eqref{eq:subprob:L1} is reduced to Problem \eqref{eq:sparse_vector_reduced_prox}.
\end{proof}

We now restate Lemma \ref{lem:low_rank_prox} and then prove it.

\begin{lemma}
Write $\mathbb{O}(k):=\{\Q\in\reals^{k\times k}:\Q^\top\Q=I_k\}$
and consider Problem~\eqref{eq:subprob:NN}.
For $\u\in\reals^s$, let $\operatorname{diag}(\u,\mathbf{0})\in\E$
denote the diagonal matrix obtained by padding $\u$ with zeros,
rectangular when appropriate.

If $\E=\reals^{m\times n}$, assume
$\R(\U\x\V^\top)=\R(\x)$ for all
$\x\in\E, \U\in\mathbb{O}(m), \V\in\mathbb{O}(n)$.
Let $\U_s,\V_s$ contain orthonormal left and right singular vectors
corresponding to the $s$ largest singular values of $\w$.
An optimal solution is
$\v^*=\U_s\operatorname{diag}(\u^*)\V_s^\top$, where
\begin{align}
\u^*\in\arg\min_{\u\in\reals^s}
\left\{
\lambda\sum_{i=1}^s(\u_i-\sigma_i(\w))^2
+\R\bigl(\operatorname{diag}(\u,\mathbf{0})\bigr)
\right\}.
\label{eq:low_rank_reduced_prox}
\end{align}

If $\E=\mbS^n$, assume
$\R(\Q\x\Q^\top)=\R(\x)$ for all
$\x\in\E, \Q\in\mathbb{O}(n)$,
and either $\operatorname{dom}(\R)\subseteq\mbS_+^n$
or $\R$ is invariant to eigenvalue sign-flipping.
Let $\theta_1,\ldots,\theta_s$ be the $s$ largest eigenvalues
of $\w$ in the former case, or the $s$ largest-in-absolute-value
eigenvalues in the latter, and let $\Q_s$ contain corresponding
orthonormal eigenvectors. An optimal solution is
$\v^*=\Q_s\operatorname{diag}(\u^*)\Q_s^\top$, where
\begin{align}
\u^*\in\arg\min_{\u\in\reals^s}
\left\{
\lambda\sum_{i=1}^s(\u_i-\theta_i)^2
+\R\bigl(\operatorname{diag}(\u,\mathbf{0})\bigr)
\right\}.
\label{eq:symmetric_low_rank_reduced_prox}
\end{align}
Ties may be broken arbitrarily.
\end{lemma}

\begin{proof}
Let $\omega\in\reals^d$ contain the singular values of $\w$
in nonincreasing order in the rectangular case, or its eigenvalues
in nonincreasing order in the symmetric case. Fix a full decomposition
\[
\w=\U\operatorname{diag}(\omega)\V^\top
\quad\text{or}\quad
\w=\Q\operatorname{diag}(\omega)\Q^\top,
\]
respectively. Choose these decompositions so that $\U_s$ and $\V_s$
consist of columns of $\U$ and $\V$, respectively, or $\Q_s$
consists of columns of $\Q$. Diagonal matrices are rectangular when appropriate.

For $\a\in\reals^d$, define
\[
\rho(\a):=\R(\operatorname{diag}(\a)),
\qquad
\mathcal T(\a):=
\begin{cases}
\U\operatorname{diag}(\a)\V^\top,
& \E=\reals^{m\times n},\\
\Q\operatorname{diag}(\a)\Q^\top,
& \E=\mbS^n.
\end{cases}
\]

Let $\v\in\operatorname{dom}(\R)\cap\mE_s$, and let $\b$
contain its singular values or eigenvalues, respectively,
in nonincreasing order.
The matrices $\v$ and $\mathcal T(\b)$ have the same rank and,
by orthogonal invariance, the same value of $\R$.
The trace inequality for singular values or eigenvalues gives
\[
\begin{aligned}
\|\v-\w\|_2^2
&=\|\b\|_2^2+\|\omega\|_2^2-2\langle\v,\w\rangle\\
&\geq\|\b\|_2^2+\|\omega\|_2^2-2\langle\b,\omega\rangle \\
&=\|\b-\omega\|_2^2
=\|\mathcal T(\b)-\w\|_2^2.
\end{aligned}
\]
Thus, restricting Problem \eqref{eq:subprob:NN} to matrices of the form
$\mathcal T(\a)$ does not change its optimal value.
For every $\a\in\reals^d$,
\[
\rank(\mathcal T(\a))=\nnz(\a),
\qquad
\R(\mathcal T(\a))+\lambda\|\mathcal T(\a)-\w\|_2^2
=\rho(\a)+\lambda\|\a-\omega\|_2^2.
\]
Consequently, Problem~\eqref{eq:subprob:NN} has the same optimal
value as
\begin{align}
\min_{\a\in\reals^d:\,\nnz(\a)\leq s}
\left\{\rho(\a)+\lambda\|\a-\omega\|_2^2\right\},
\label{eq:lem:low_rank_prox:vecProb}
\end{align}
and every minimizer $\a^*$ of this vector problem corresponds to a matrix
minimizer $\mathcal T(\a^*)$.

The function $\rho$ is permutation invariant.
In the rectangular case it is also invariant to coordinate-wise
sign flips, and in the symmetric case it either has nonnegative domain
or is invariant to coordinate-wise sign flips.
Let $I=\{i_1,\ldots,i_s\}$ be the column indices of the selected
vectors in the above decompositions, ordered as in the statement,
and write
\[
\P_I=[\e_{i_1},\ldots,\e_{i_s}],
\qquad
\omega_I=\P_I^\top\omega.
\]
Lemma~\ref{lem:sparse_vector_prox} gives an optimizer
$\a^*=\P_I\u^*$ of~\eqref{eq:lem:low_rank_prox:vecProb}, where
\begin{align}
\u^*\in\arg\min_{\u\in\reals^s}
\left\{\rho(\P_I\u)+\lambda\|\u-\omega_I\|_2^2\right\}.
\label{eq:lem:low_rank_prox:vecProbReduced}
\end{align}
By permutation invariance,
$\rho(\P_I\u)=\R\bigl(\operatorname{diag}(\u,\mathbf0)\bigr)$.
Since $\omega_I$ contains $\sigma_1(\w),\ldots,\sigma_s(\w)$
or $\theta_1,\ldots,\theta_s$, respectively,
\eqref{eq:lem:low_rank_prox:vecProbReduced} is precisely
\eqref{eq:low_rank_reduced_prox} or
\eqref{eq:symmetric_low_rank_reduced_prox}.
Finally, the corresponding matrix minimizer is
\[
\mathcal T(\P_I\u^*)
=\U_s\operatorname{diag}(\u^*)\V_s^\top
\quad\text{or}\quad
\mathcal T(\P_I\u^*)
=\Q_s\operatorname{diag}(\u^*)\Q_s^\top,
\]
respectively, as claimed.
\end{proof}

\bibliography{bibs.bib}
\bibliographystyle{plain}

\end{document}